\documentclass[12pt]{amsart}
\usepackage{amsmath}
\usepackage{amsthm}
\usepackage{amsfonts}
\usepackage{amssymb}
\usepackage{verbatim}
\usepackage{graphicx}
\usepackage[margin=1.0in]{geometry}
\usepackage{url}
\usepackage{enumerate}
\usepackage[pdftex,bookmarks=true]{hyperref}

\newcommand\diag{\operatorname{diag}}

\newcommand\R{{\mathbb{R}}}
\newcommand\C{{\mathbb{C}}}
\newcommand\Z{{\mathbb{Z}}}
\newcommand\N{{\mathbb{N}}}
\newcommand\F{{\mathbb{F}}}

\newcommand\I{{\mathbf{I}}}
\renewcommand\P{{\mathbf{P}}}
\newcommand\E{{\mathbf{E}}}

\renewcommand\Re{{\operatorname{Re}}}
\newcommand\eps{{\varepsilon}}
\newcommand\Trace{\operatorname{trace}}
\newcommand\trace{\operatorname{trace}} 
\newcommand\Irr{\operatorname{Irr}}

\newcommand\Ker{\operatorname{Ker}}

\newcommand\BC{{\mathbf C}}

\newcommand\BZ{{\mathbf Z}}

\newcommand\ord{{\mathbf {o}}} 

\newcommand\CP{{\mathcal P}}

\theoremstyle{plain}
  \newtheorem{theorem}{Theorem}[section]

  \newtheorem{proposition}[theorem]{Proposition}
  
  \newtheorem{lemma}[theorem]{Lemma}
  \newtheorem{corollary}[theorem]{Corollary}

\theoremstyle{definition}

  \newtheorem{remark}[theorem]{Remark}

\numberwithin{equation}{section}

\include{psfig}

\begin{document}

\title {Non-abelian Littlewood-Offord inequalities}

\author{Pham H. Tiep} 
\address{Department of Mathematics, University of Arizona, Tucson, AZ 85721, USA}
\thanks{P. H. Tiep gratefully acknowledges the support of the NSF (grant DMS-1201374) and the Simons Foundation
Fellowship 305247.}
\email{tiep@math.arizona.edu}

\author{Van H. Vu}
\thanks{V. H. Vu is supported by research grants DMS-0901216 and AFOSAR-FA-9550-09-1-0167. }
\address{Department of Mathematics, Yale University, New Haven , CT 06520, USA}
\email{van.vu@yale.edu}


\begin{abstract}  In 1943, Littlewood and Offord proved the first  anti-concentration result for sums of independent random variables. 
Their result has since then been strengthened and generalized by generations of researchers, with applications in several areas of mathematics.

In this paper, we present the first non-abelian analogue of Littlewood-Offord result, a sharp anti-concentration inequality for products of 
independent random variables. 

\end{abstract}

\maketitle

\section{Introduction} \label{section:introduction} 

In 1943, motivated by their studies of random polynomials, Littlewood and Offord
\cite{LO} proved a remarkable fact about the distribution of a  sum of independent random variables.   Let $V$ be  a sequence of (not necessarily different) non-zero real numbers  $a_1, \dots, a_n$ and set

$$\rho_V :=  \sup_{b  \in \R}  \P( \sum_{i=1} ^n \hat a_i  = b) , $$ 
where the $\hat {a_i}  $ are independent  random variables taking values $\pm a_i$ with probability $1/2$. 

\begin{theorem} $$\rho_V    = O(  n^{-1/2 } \log n ) .$$ 
\end{theorem} 

\noindent Here and hereafter, the asymptotic notation is used under the assumption that $n \rightarrow \infty$. Soon after their paper, 
Erd\H os \cite{E}, removing the $\log n$ term, optimized the bound.

\begin{theorem} \label{LOE} (Littlewood-Offord-Erd\H os)  $$\rho_V  \le   \frac{ { n \choose \lfloor n/2 \rfloor }} {2^n}  =O(  n^{-1/2 } ) .$$ 
\end{theorem}

The bound is sharp, as shown by taking all $a_i =1$. It is easy to see that in this case $\P (\sum_{i=1}^n \hat {a_i}  =\delta )=  \frac{ { n \choose \lfloor n/2 \rfloor }} {2^n} $, 
where $\delta=1$ if $n$ is odd and $0$ otherwise.  In \cite{Kle1}, Kleitman generalized Theorem \ref{LOE} to complex setting. 

\vskip2mm 

It is best to relate Theorem \ref{LOE} to
 the classical Berry-Esseen theorem, which asserts that if the $a_i$ are all of magnitude 1, then the distribution of $\frac{1}{\sqrt n }  \sum_{i}^n \hat {a_i } $ converges to the normal distribution with rate $O(n^{-1/2} )$. This implies that for any point $b$, $\P (\sum_{i=1} ^n \hat {a_i}  = b) = O( n^{-1/2} )$. Theorem \ref{LOE}  strengthens this fact significantly, asserting that the probability in question is  always $O(n^{-1/2} )$, {\it regardless of the magnitude of the $a_i$}. 
 
 \vskip2mm 

Theorem \ref{LOE} has become the starting point of a long line of research, which continues through several  decades and has recently becomes very active (see, for instance, the survey \cite{NV}).  It has been 
strengthened (under various conditions) and generalized in different directions  by many  researchers, including Esseen, Kolmogorov, Rogozin, 
Hal\'asz,  Stanley, Kleitman, Szemer\'edi-S\'ark\"ozy, Tao,  and others; 
see, for instance \cite{EM, Es, FF, Gr, H, Kle1, Kle2, Kle3, Ng, NgV, Rog, SSz, RV, TV1, TV2, TV3, Stan}.  These results  are often referred to 
as anti-concentration inequalities, and  have found surprising applications in  different areas of mathematics, 
including random matrix theory. For more details,  the reader may want to check  the recent survey \cite{NV}. 

\vskip2mm 

A limitation to all  existing extensions of Theorem \ref{LOE} is that they only apply for  random variables taking values in an abelian group (in most cases $\R^d $ or $\C$), as the  available proof techniques 
only work in this setting. 

\vskip2mm 

The goal of this paper is to initiate the study of the anti-concentration phenomenon in  the non-abelian setting. 
Let $V$ be  a sequence of (not necessarily distinct) non-trivial elements  $A_1, \dots, A_n$  of a non-abelian group $G$ and set 

$$\rho_V :=  \sup_{B  \in G}  \P ( \prod_{i=1}^n \hat {A_i } = B ) ,$$  
where the $\hat {A_i } $ are independent random variables taking values $A_i$ or $A_i^{-1} $ with probability $1/2$. 

In what follows, we work toward the most natural non-abelian extension of Theorem \ref{LOE} with the  $A_i$ being invertible matrices of a fixed size $m$. 
Our goal is  to bound $\rho_V$. Thanks to the famous work of Furstenberg and Kesten \cite{FK}, a central limit theorem for products of i.i.d. random matrices 
is known, and Berry-Essen type local central limit theorems are also available; see for instance \cite{LV} and the references 
therein. But in the cases where central limit theorems fail, our result provides  the first anti-concentration 
inequality. 

\vskip2mm
In view of Theorem \ref{LOE}, it is tempting to guess that $\rho_V = O(n^{-1/2})$. There are several constructions matching this bound.

\vskip2mm

{\it Example 1.} Let $A$ be an invertible $m \times m$ matrix over some field $\F$ of infinite order 
(as an element of $GL_m(\F)$), and $A_i = A$ all $1 \le i \le n$. Then 

$$\P (\prod_{i=1}^n \hat {A_i}  =\Delta )=  \frac{ { n \choose \lfloor n/2 \rfloor }} {2^n}  = \Theta (n^{-1/2} ) $$ where $\Delta= A$ if $n$ is odd and $I$ (the identity matrix) if $n$ is even.

\vskip2mm 

{\it Example 2.} Let 
$$ A_i =  \left[ \begin{array}{rr}
1 & a_i  \\
0 & 1  \end{array} \right]$$
\noindent  where $a_i$ is an integer with absolute value at most $K$, where $1 \le K= O(1)$.  We have 

$$A_i A_j = \left[ \begin{array}{rr}
1 & a_i  + a_j \\
0 & 1  \end{array} \right],~~ 
A_i ^{-1} = \left[ \begin{array}{rr}
1 & -a_i  \\
0 & 1  \end{array} \right].  $$

The top right corner of the random product $\prod_{i=1}^n \hat {A_i}$  is the random sum $S := \sum_{i=1}^n \hat {a_i}$, which has mean 0 and variance 
$\sum_{i=1}^n a_i^2 \le nK^2 $.  By Chebyshev's inequality, with probability at least $3/4$, $|S| \le 2 K n^{1/2}$. 
By the pigeonhole principle, there is an integer $s$ such that 

$$ \P(S=s) \ge \frac{3}{4} \cdot  \frac{1}{2 K n^{1/2} +1} \ge \frac{1}{4K n^{1/2}  }. $$

\noindent It   follows that 

$$\rho_V  \ge  \frac{1}{4K} n^{-1/2} . $$

\vskip2mm 


However, 
the presence of  torsions makes the problem more subtle than its abelian counterparts. Assume that all $A_i =A$ and $A$ has order $s$, then the product 
$\prod_{i=1}^n \hat A_i $ can take only $s$ values $ A, A^2, \dots, A^{s-1}, A^s=I $ (where $I$ again denotes the identity matrix). 
 By the pigeonhole principle, it follows that 

$$\rho_ V  \ge \frac{1}{s}. $$

\vskip2mm

We are now ready to state our main results, which are sharp with respect to the above examples.

\begin{theorem} \label{theorem:main}
For any integers $m, n, s \geq 2$  the following statement holds. 
Let  $V$ be a sequence  $A_1, \dots, A_n$ of elements of $GL_m (\C)$ of  order at least $s$. 
Then 
$$\rho_V  \le 141\cdot \max \left\{\frac{1}{s}, \frac{1}{\sqrt n }  \right\}.$$
\end{theorem}



Our approach also yields the following theorem.

\begin{theorem} \label{theorem:main0}  
Let  $m,n, s \geq 2$ be integers  and let $p \geq \min\{s,\sqrt{n}\}$ be a sufficiently large prime. Let 
   $V$ be a sequence  $A_1, \dots, A_n$ of elements of $GL_m (p)$  of  order at least $s$. 
Then 
$$\rho_V  \le  141\cdot  \max \left\{\frac{1}{s}, \frac{1}{\sqrt n }  \right\} .$$ 
\end{theorem}


We can generalize Theorem \ref{theorem:main} to the following

\begin{theorem} \label{theorem:main2}
For any integers $m, n,s  \geq 2$ the following statement holds. 
Let  $V$ be a sequence  $A_1, \dots, A_n$ of elements of $GL_m (\C)$  where at least $N$ of them have  order at least $s$. 
Then 
$$ \rho_V  \le 141 \cdot \max \left\{\frac{1}{s}, \frac{1}{\sqrt N}  \right\} .$$
\end{theorem}

\begin{remark}  \label{remark:main} The above theorems  also hold for all groups which can be embedded into 
 $GL_m (p)$ or  $GL_m (\C)$.    The word {\it embed} can be interpreted in two ways. First, one can use canonical embeddings, such as 
 embeddings of simple Lie groups into $GL_m (\C)$, or more generally, embeddings via faithful finite-dimensional 
 complex representations. In particular, Theorem \ref{theorem:main2} holds for any arbitrary group $G$ that admits a faithful finite-dimensional complex representation. Next, one can also use the embeddings discussed in Section \ref{section:embedding}.   The reader is invited to work out an example. 
 The term $141$ can be  improved somewhat, but we do not try to push in this direction.  \end{remark}

\vskip2mm 

Our proof combines tools from three different areas: additive combinatorics, representation theory and linear algebra, and 
we also use results from analytic number theory.
First, Section \ref{section:embedding}  provides us with Freiman-like embeddings that allow us to map our problems from infinite settings to a finite setting of $GL_m(p)$, with the prime $p$ chosen
suitably. Next, in Section \ref{section:representations} we prove the key Proposition 
\ref{pro:rep2}  which shows that, if $m$ or $p$ is sufficiently large and $\Phi$ is an 
irreducible complex representation of $SL_m(q)$ (with $p|q$), then the eigenvalues of 
$\Phi(g)$ for any non-central element $g \in SL_m(q)$ have an ``almost'' uniform distribution. 
In Sections \ref{section:trace} and \ref{section:singular}, we give a 
representation-theoretic formula for the probability in question, and provide a way to bound it
from the above by using estimates on singular values.
The proofs of the main results will then be presented in Section \ref{section:proof}. 

\vskip2mm 

{\bf Notation.} The asymptotic notation is used under the assumption that $n \rightarrow \infty$. For a group $G$,  ${\bf Z } (G)$ is its center, $\Irr(G)$ denotes the set of isomorphism classes of its irreducible representations 
(or the set of its complex irreducible characters, depending on the context), $\langle X\rangle$ denotes the subgroup generated by 
a subset $X$ of $G$, ${\bf 1}$ denotes the identity element. In the matrix setting, $I$ denotes the identity matrix. For an element $g \in G$, we denote by $\ord(g)$ its order, and 
$\BC_G(g)$ its centralizer. If $\alpha$ and $\beta$ are complex characters of a finite group $G$, then 
$[\alpha,\beta]_G$ denotes their scalar product, and $\alpha_H$ denotes the restriction of $\alpha$ to a subgroup
$H \leq G$.
$\P, \E, \I_E$ denote probability, expectation, indicator variable of an event $E$, respectively.


\section{Embedding theorems} \label{section:embedding} 

In this section, we discuss results  that allow us to map our problem from an infinite setting (the underlying group is infinite)  to a finite setting (the underlying group is finite) and vice versa. 
  Let us start with  the following result, which is a special case of \cite[Theorem 1.1]{VWW}.

\begin{theorem} \label{theorem:embed} 
Let $S$ be a finite collection of complex numbers and $L$ a finite collection of non-zero elements of $\Z [S]$. Then there is an infinite sequence $\CP$  of primes $p$  such that for any $p \in \CP$, 
there is a ring  homomorphism 
$\Phi$ from $\Z [S]$ to $\Z/p \Z$ such that $0 \notin \Phi (L)$.
\end{theorem}

We obtain the following corollary. 

\begin{corollary}  \label{lemma:embed} 
Let $A_1, \dots, A_n$ be (not necessarily distinct) elements of $GL_m (\C)$. Then there is an infinite sequence $\CP$ of primes
$p$ such that for any $p \in \CP$, 
there is a group  homomorphism $\Phi$ from 
$\langle A_1, \dots, A_n \rangle$ to $GL_m (p) $ such that 

\begin{enumerate}[\rm(i)] 

\item  If  $\ord(A_i) < \infty$, then $\ord(\Phi (A_i)) = \ord(A_i)$;

\item If $\ord(A_i)= \infty$,  then $\ord(\Phi (A_i)) \ge n$.

\end{enumerate} 

\end{corollary}

\begin{proof} Let $S \subset \C$ be the set consisting of all the entries of the $A_i$ and their complex conjugates. 
For each $i$, let $d_i = \min \{ n, \ord(A_i) \}$.  For each $1 \le j \le d_i -1$, $1 \le k,l \le m$,  let  $F_{i,j,k,l}$ be the $(k,l)$-th entry of $A_i^j$.
It is clear that  both $F_{i,j,k,l }$  and $\bar F_{i,j,k,l}$ are   polynomials in $S$ with integer coefficients. Define 

$$G(i,j) := \sum_{1 \le k \neq l \le m}  F_{i,j,k,l} \bar F_{i,j,k,l } + \sum_{k=1}^m (F_{i,j,k,k}- 1) ( \bar F_{i,j,k,k }-1) . $$

Again $G(i,j)$ is a polynomial in $S$ with integer coefficients. More importantly, $G(i,j) =0 $ iff $A_i^j = I$. 
Now we apply Theorem \ref{theorem:embed} with $L$ being the collection of the 
$G(i,j)$  for all possible  pairs $i,j$.  The map from $\langle A_1, \dots, A_n \rangle$ to $GL_m (p)$ is induced trivially by the map from 
$\Z [S] $ to $\Z/ p\Z$. 
\end{proof}

\section{Representation theory}\label{section:representations}

Our study will make use of several non-trivial facts about the  irreducible representations of a finite group, the most critical 
one being an estimate on  the multiplicities of the eigenvalues. Let us start with a toy  lemma.

\begin{lemma}\label{lemma:mult1}
Let $G$ be a finite group and let $\Phi$ be a complex irreducible representation of $G$ with character 
$\chi$. Suppose there is a constant $0 < \alpha < 1$ such that $|\chi(x)/\chi(1)| \leq \alpha$ for all 
$x \in G \smallsetminus \BZ(G)$. If $g \in G \smallsetminus \BZ(G)$ is such that 
$g\BZ(G)$ has order $k_1$ in $G/\BZ(G)$, then the multiplicity $m$ of any eigenvalue of $\Phi(g)$ satisfies
$$(k_1^{-1}-\alpha)(\dim \Phi) < m < (k_1^{-1}+\alpha)(\dim \Phi).$$  
\end{lemma}

\begin{proof}
Let $k$ be the order of $g$ in $G$ and let $\eps \in \C$ be a primitive $k$-th root of unity. Then $k_1\mid k$ and
$g^{k_1} \in \BZ(G)$. By Schur's Lemma, $\Phi(g^{k_1}) = \eps^{k_1l}I$ for some $l \in \N$. It follows that if $\lambda$
is any irreducible constituent of the character $\chi$ of $\Phi$ restricted to $C := \langle g \rangle$, then  
$\lambda(g^{k_1}) = \eps^{k_1l}$ and $\chi(g^i) = \chi(1)\lambda(g^i)$ if $k_1|i$. By assumption, 
$|\chi(g^i)| \leq \alpha \chi(1)$ if $k_1 \nmid i$. Now,
$$[\chi_C,\lambda]_C = \left|\frac{1}{k}\sum^{k-1}_{i=0}\chi(g^i)\bar\lambda(g^i)\right| \leq  
    \frac{1}{k}\left( \left|\sum_{k_1|i}\chi(g^i)\bar\lambda(g^i)\right|  + \sum_{k_1 \nmid i}|\chi(g^i)\bar\lambda(g^i)|\right)$$
$$\leq \frac{\chi(1)}{k}\left( \frac{k}{k_1} + \alpha(k-k_1)\right)< \chi(1)(k_1^{-1} + \alpha).$$ 
The lower bound $[\chi_C,\lambda]_C >\chi(1)(k_1^{-1} - \alpha)$ is proved similarly.
\end{proof}


\begin{corollary}\label{cor:mult2}
Let $G = SL_d(q)$ with $d \geq 2$ and $q \geq 49$. Suppose that $\Phi$ is a 
complex irreducible representation of degree $>1$ of $G$ and $g \in G \smallsetminus \BZ(G)$ is such that 
$g\BZ(G)$ has order $k_1$ in $G/\BZ(G)$. Then the multiplicity $m$ of any eigenvalue of $\Phi(g)$ satisfies
$$(k_1^{-1}-\alpha)(\dim \Phi) < m < (k_1^{-1}+\alpha)(\dim \Phi)$$
with $\alpha := 1/(\sqrt{q}-1)$.  
\end{corollary}

\begin{proof}
By Theorems 3.3 and 5.3 of \cite{G}, $G$ satisfies the hypothesis of Lemma \ref{lemma:mult1} with the specified $\alpha$. 
Hence the statement follows.
\end{proof}

Notice that if $g\BZ(G)$ has order $k_1$ in $G/\BZ(G)$, then its eigenvalues are among the $k_1$-th roots of unity of 
a fixed complex number.
Corollary \ref{cor:mult2}  asserts that each eigenvalue appears about $(\dim \Phi)/ k_1 $ times. 
However, this result is not sufficiently strong for our purposes  if the term $\alpha$ dominates. 
In applications, we will need the following more precise statement which we formulate for both 
$G = SL_d(q)$ and $S = PSL_(q) = G/\BZ(G)$.

\begin{proposition}\label{pro:rep2}
Let $H \in \{G,S\}$ where $G = SL_d(q)$ and $S = PSL_d(q)$ with $d \geq 3$ and $q \geq 11$. Set 
$e := \gcd(d,q-1)$ and $e_H := |\BZ(H)|$ so that $e_G = e$ and $e_S = 1$.

\begin{enumerate}[\rm(i)]
\item $|\Irr(G)| \leq q^{d-1} + 3q^{d-2} \leq (14/11)q^{d-1}$ and $|\Irr(S)| \leq (q^{d-1} + 5q^{d-2})/e \leq (16/11e)q^{d-1}$. Furthermore,
$$\sum_{\chi \in \Irr(H),~\chi(1) < q^{(d^2-d-1)/2}}\chi(1)^2 <  \frac{5}{3q} \cdot |H|.$$

\item Suppose that $\chi \in \Irr(H)$ satisfies $\chi(1) \geq q^{(d^2-d-1)/2}$ and $x \in H \setminus \BZ(H)$. Then 
$$\frac{|\chi(x)|}{\chi(1)} \leq q^{(3-d)/2}.$$

\item Suppose that $q = p^f$ is a power of a prime $p \geq d$. Then for any $x \in H$, we have 
that $\ord(x) < q^{d+1}$; furthermore, either $|\BC_H(x)| \leq q^{d^2-4d+8}$ or $\ord(x) < q^2$. 

\item Suppose that $d \geq 43$, $q = p \geq d$, and $p$ or $d$ is chosen sufficiently large. 
For any element $g \in H \smallsetminus \{1\}$ of order say $N$, if 
$\Phi$ is a complex irreducible representation of $H$ of degree $\geq q^{(d^2-d-1)/2}$, then any eigenvalue of 
$\Phi(g)$ occurs with multiplicity at most $(e_H+2)(\dim \Phi)/N$. 
\end{enumerate}

\end{proposition}

\begin{proof}
(i) The first two inequalities are just Proposition 3.6(2) and Corollary 3.7(2) of \cite{FG}. Next, by \cite[Lemma 4.1(i)]{LMT} we have 
$$\frac{|G|}{q^{d^2-1}} > \prod^{\infty}_{i=1}(1-q^{-i}) > 1-\frac{1}{q}-\frac{1}{q^2} \geq \frac{109}{121}$$
as we assume $q \geq 11$. It follows that 
$$\frac{1}{|G|}\sum_{\chi \in \Irr(G),~\chi(1) < q^{(d^2-d-1)/2}}\chi(1)^2 <
    \frac{|\Irr(G)| \cdot q^{d^2-d-1}}{|G|} < \frac{(14/11) \cdot q^{d^2-2}}{(109/121) \cdot q^{d^2-1}} < \frac{3}{2q}.$$
Next, as $|S| = |G|/e$, we have     
$$\frac{1}{|S|}\sum_{\chi \in \Irr(S),~\chi(1) < q^{(d^2-d-1)/2}}\chi(1)^2 <
    \frac{|\Irr(S)| \cdot q^{d^2-d-1}}{|S|} < \frac{(16/11e) \cdot q^{d^2-2}}{(109/121e) \cdot q^{d^2-1}} < \frac{5}{3q}.$$
    
(ii) As mentioned in the proof of \cite[Proposition 6.2.1]{LST}, $|\BC_G(x)| \leq q^{d^2-2d+2}$. As 
$|\chi(x)| \leq |\BC_G(x)|^{1/2}$, the claim follows for $H = G$. The claim for $H = S$ follows in a similar fashion, by noting
that 
\begin{equation}\label{cent-gs}
  |\BC_S(x)| \leq |\BC_G(\hat{x})|
\end{equation}
if $x \in S = G/\BZ(G)$ and $\hat{x}$ is an inverse image of $x$ in $G$. (Indeed, let $D$ be the complete inverse image of 
$\BC_S(x)$ in $G$. Then, for any $g \in D$ we have $g\hat{x}g^{-1} = \hat{x}f(g)$ for some element 
$f(g) \in \BZ(G)$. It is easy to see that $f \in {\mathrm {Hom}}(D,\BZ(G))$ with $\Ker(D) = \BC_G(\hat{x})$. Hence,
$|D| \leq |\BZ(G)| \cdot |\BC_G(\hat{x})|$, and so
$$|\BC_S(x)| = |D/\BZ(G)| \leq |\BC_G(\hat{x})|.)$$

\smallskip
(iii) First we prove the claims for $H = G$.
Write $x = su$ as a commuting product of a semisimple element $s \in G$ and a unipotent element 
$u \in G$; in particular, $\ord(x) = \ord(s) \cdot \ord(u)$. The condition $p \geq d$ implies that 
$$\ord(u) \leq p \leq q.$$ 
Next, we consider the characteristic polynomial $f(t) \in \F_q[t]$ of the semisimple linear transformation $s$ of 
the vector space $V = \F_q^d$. We also decompose $f(t) = \prod^m_{i=1}f_i(t)^{k_i}$ as a product of powers 
of pairwise distinct monic irreducible polynomials $f_i \in \F_q[t]$, with $\deg f_i = a_i$, so that
$\sum^m_{i=1}k_ia_i = d$. Then both $s$ and $\BC_G(s)$ preserve a direct sum decomposition
$V = \oplus^m_{i=1}V_i$ where $\dim V_i = k_ia_i$, and the action of $s$ on $V_i$ has order at most $q^{a_i}-1$. 
It follows that
$$\ord(s) \leq \prod^m_{i=1}(q^{a_i}-1) \leq q^{\sum^m_{i=1}a_i}-1 \leq q^d-1,$$
implying $\ord(x) = \ord(s) \cdot \ord(u) < q^{d+1}$ as stated. We also note that 
\begin{equation}\label{cent}  
  \BC_G(s) \hookrightarrow GL_{k_1}(q^{a_1}) \times GL_{k_2 }(q^{a_2}) \times \ldots \times  GL_{k_m}(q^{a_m}).
\end{equation}

Suppose in addition that $|\BC_G(x)| > q^{d^2-4d+8}$. As the semisimple part $s$ of $x$ is a power of $x$,
we also have that $|\BC_G(s)| > q^{d^2-4d+8}$. It suffices to show that in this case $\ord(s) \leq q-1$. 
We may assume that 
\begin{equation}\label{cent2}
  k_1a_1 \geq k_2a_2 \geq \ldots \geq k_ma_m.
\end{equation}  
If $m = 1$ and $a_1 \geq 2$, then \eqref{cent} implies that 
$$|\BC_G(s)| \leq |GL_{k_1}(q^{a_1})| < q^{k_1^2a_1} = q^{d^2/a_1} \leq q^{d^2/2} \leq q^{(d^2-4d+8)},$$
contrary to the assumption. On the other hand, if $a_1 = 1$, then $k_1 = d$ and so $s \in \BZ(G)$ and
$\ord(s) \leq q-1$ as desired. So we may assume that $m \geq 2$. Suppose now that $k_1a_1 =: k \geq 2$
and $d-k \geq 2$. Then, as $\BC_G(s)$ preserves the decomposition $V = \oplus^m_{i=1}V_i$, we have that
$$|\BC_G(s)| \leq |GL(V_1) \times GL(\oplus^m_{i=2}V_i)|
    = |GL_k(q) \times GL_{d-k}(q)| < q^{k^2+(d-k)^2} \leq q^{d^2-4d+8},$$
again a contradiction. If $k = 1$, then the choice \eqref{cent2} implies that $k_ia_i = 1$ for all $i$ and 
so $\BC_G(s) \leq GL_1(q)^d$ by \eqref{cent}, whence $\ord(s) \leq q-1$. So we may assume $k =d-1$. If furthermore
$a_1 \geq 2$, then again by \eqref{cent} we have 
$$|\BC_G(s)| \leq |GL_{k_1}(q^{a_1}) \times GL_1(q)| < q^{k_1^2a_1+1} = q^{(d-1)^2/a_1+1} \leq 
    q^{(d^2-2d+3)/2} < q^{d^2-4d+8},$$
contrary to the assumption. Thus $(m,k_1,a_1,k_2,a_2) = (2,d-1,1,1,1)$, in which case we again have $\ord(s) \leq q-1$.

\smallskip
Now we prove the claims for $H = S$. Let $x \in S = G/\BZ(G)$ and let $\hat{x}$ be an inverse image of $x$ in $G$. 
Then it is clear that 
\begin{equation}\label{order-gs}
  \ord(x) \leq \ord(\hat{x}).
\end{equation}
As $\ord(\hat{x}) \leq q^{d+1}$, we also have $\ord(x) \leq q^{d+1}$. Next, we have shown that either 
$|\BC_G(\hat{x})| \leq q^{d^2-4d+8}$, in which case $|\BC_S(x)| \leq q^{d^2-4d+8}|$ by \eqref{cent-gs}, or
$\ord(\hat{x}) \leq q^2$, in which case $\ord(x) \leq q^2$ by \eqref{order-gs}.     
   
\smallskip
(iv)  Let $M$ denote the largest multiplicity of any eigenvalue of $\Phi(g)$ and let $N_1$ denote the order of $x\BZ(H)$
in $H/\BZ(H)$. Then $x^{N_1} \in \BZ(H)$ and so $x^{N_1e_H} = 1$. It follows that $N|N_1e_H$; in particular,
$N_1 \geq N/e_H$. Now, if $N \leq q^{(d-3)/2}$, then (ii) and the proof of Lemma  
\ref{lemma:mult1} (with $\alpha := q^{(3-d)/2}$) show that 
$$M \leq (\dim \Phi)(q^{(3-d)/2} + 1/N_1) \leq (\dim \Phi)(q^{(3-d)/2} + e_H/N)\leq (e_H+1)(\dim \Phi)/N.$$

Now we may assume that $N > q^{(d-3)/2}$. Consider and any element $g^i \notin \BZ(H)$. Note by (iii) that $N < q^{d+1}$.
If $|g^i| \geq q^2$, then by (iii) we have that
\begin{equation}\label{cent3}
  |\chi(g^i)| \leq |\BC_H(g^i)|^{1/2} \leq q^{(d^2-4d+8)/2} < (\dim \Phi)/q^{d+1} < (\dim \Phi)/N.
\end{equation}  
On the other hand, if $|g^i| < q^2$, then, as $d \geq 43$, $|g^i| \leq N^{0.1}$. The number $L$ of such elements 
$g^i$ is at most the sum of all divisors of $N$ that do not exceed $N^{0.1}$. If $p$ or $d$ is chosen sufficiently large, then
$N > p^{(d-3)/2}$ is large enough, so that the total number of divisors of $N$ is at most $N^{0.2}$, cf. \cite[p. 296]{Ap}. It follows that 
\begin{equation}\label{cent4}  
  L \leq N^{0.1} \cdot N^{0.2} = N^{0.3}.
\end{equation}  
Certainly, as $g^i \notin \BZ(H)$ and $N < q^{d+1}$ we have by (ii) that   
\begin{equation}\label{cent5}  
  |\chi(g^i)|/\chi(1) \leq q^{(3-d)/2} < N^{-1/3}.
\end{equation}  
Now we can follow the proof of Lemma \ref{lemma:mult1} and obtain by \eqref{cent3}--\eqref{cent5} that
the multiplicity of any irreducible constituent $\lambda$ of $\chi_C$ is
$$[\chi_C,\lambda]_C = \left|\frac{1}{N}\sum^{N-1}_{i=0}\chi(g^i)\bar\lambda(g^i)\right|$$ 
$$\leq  \frac{1}{N}\left( \sum_{g^i \in \BZ(H)}|\chi(g^i)\bar\lambda(g^i)| + \sum_{g^i \notin \BZ(H),~|g^i| < q^2}|\chi(g^i)\bar\lambda(g^i)|  + \sum_{|g^i| \geq q^2}|\chi(g^i)\bar\lambda(g^i)|\right)$$
$$\leq \frac{\chi(1)}{N}\left(e_H+ \frac{L-e_H}{N^{1/3}} + \frac{N-L}{N}\right)< \frac{(e_H+2)\chi(1)}{N},$$
if $\chi$ is the character of $\Phi$ and $C = \langle g \rangle$.
It follows that $M \leq (e_H+2)(\dim \Phi)/N$. 
\end{proof}

\section {A trace identity}\label{section:trace} 

Let $G$ be a finite group. Let $A_1, \dots, A_n, B$ be (not necessarily distinct) elements of $G$. As usual, $\hat {A_i}$ denotes the random variable taking values $A_i$ and $A_i^{-1} $ with probability $1/2$.
The following identity plays an important role in our proof.

\begin{lemma} 
\begin{equation} \label{key}  
\P (\prod_{i=1}^n \hat {A_i } = B) = \frac{1}{|G|}\sum_{\Phi \in \Irr(G)}(\dim \Phi)  \trace \left(\prod_{i=1}^n \frac{ \Phi (A_i) + \Phi (A_i^{-1})  }{2} \Phi (B^{-1})\right) .
\end{equation} 

\end{lemma}

\begin{proof}  Consider the regular representation $R$ of $G$. It is  the direct sum of 
$\dim(\Phi)$ copies of each $\Phi \in \Irr(G)$ and for any $X \in G$,  
$ \Trace(R(X))$ equals $|G|$ if $X = {\bf 1}  $ and 
$0$ otherwise.   In other words,

$$ \I_{X= {\bf 1} } = \frac{1}{|G|} \Trace (R (X) ). $$ 

Now let $X:= \prod_{i=1}^n \hat {A_i }  B^{-1} $. We have 

$$\P (\prod_{i=1}^n \hat {A_i } = B) = \E I_{X ={\bf 1}  }  = \frac{1}{|G|} \E \Trace R(X). $$

By linearity of expectation, 

$$\E \Trace R(X) = \E  \sum_{\Phi \in \Irr(G) } (\dim \Phi)  \Trace \Phi (X)  = \sum_{\Phi \in \Irr (G) } (\dim \Phi) \E \Trace \Phi (X). $$

Furthermore, by linearity of trace 

$$\E \Trace \Phi (X)  = \Trace \E \Phi (X) = \Trace \E \prod_{i=1}^n  \Phi (\hat {A_i} ) \Phi (B^{-1}) . $$

As the $\hat {A_i}$ are independent, 

$$ \E \prod_{i=1}^n  \Phi (\hat {A_i} ) \Phi (B^{-1} ) = \prod_{i=1}^n \E \Phi (\hat {A_i}) \Phi (B^{-1}) = \prod_{i=1} ^n \frac{\Phi (A_i) + \Phi (A_i^{-1} ) }{2} \Phi (B^{-1})  , $$ concluding the proof. 
\end{proof}

\section{Singular value estimates}\label{section:singular} 

For any complex $n \times n$-matrix $M$, let $s_1(M) \geq s_2(M) \geq \ldots \geq s_n(M)$ denote the 
singular values of $M$ (listed in non-increasing order). 
The evaluation of \eqref{key} relies on several singular value estimates, which we collect in the following lemma. 

\begin{lemma}  \label{lemma:singular} Let $M, M'$ be square matrices of size $n$. Then

$$| \Trace M |  \le \sum_{i=1} ^n s_i (M). $$

\noindent Furthermore, for any $ 1\le k \le n$,

$$s_k ( MM') \le \min\{s_k (M) s_1 (M'),s_1(M)s_k(M')\},$$ and 

$$\prod_{j=1} ^k s_j (MM') \le \prod_{j=1}^k s_j (M) \prod _{j=1}^k s_j (M') . $$

\end{lemma}

\begin{proof}  
To prove the first statement, notice that  the singular value decomposition yields $M = UDV$, where $U = (u_{ij})$ and $V = (v_{ij})$ are unitary and 
$$D = \diag(s_1(M), s_2(M), \ldots ,s_n(M)).$$
Next, for a given $j$, 
$$|\sum^n_{i=1} u_{ij}v_{ji}| \leq (\sum^n_{i=1} |u_{ij}|^2)^{1/2}(\sum^n_{i=1} |v_{ji}|^2)^{1/2} = 1.$$
Hence,
$$|\Trace M| = |\sum^n_{i,j=1}u_{ij}s_j(M)v_{ji}| \leq \sum^n_{j=1} s_j(M)|\sum^n_{i=1} u_{ij}v_{ji}| = \sum^n_{j=1} s_j(M).$$

\smallskip

The second statement follows from the min-max definition of singular values \cite[Problems III.6.1, III.6.2]{Bha}.

\vskip2mm 
To prove the third, consider the wedge product  $\wedge^k M$. We have by the second statement

$$s_1 (\wedge^k (MM')) \le s_1 (\wedge^k M) s_1 (\wedge ^k  M' ). $$

We obtain the claim by combining this  with the well-known fact \cite[page 18]{Bha} that   $s_1 (\wedge^k A) =\prod_{j=1}^k s_j (A)$ for any matrix $A \in GL_n(\C)$.  \end{proof}

We will also need the following two elementary facts. 

\begin{lemma} \label{lemma:unitary} Let $U$ be a unitary matrix of size $d$ with eigenvalues $\lambda_1, \dots, \lambda_d$. Then the singular values of $(U + U^{-1} )/2$ are 
$|\Re(\lambda_1)| , \dots,  |\Re(\lambda_d)| $. \end{lemma} 

\begin{lemma}\label{lemma:trig}
\begin{enumerate}[\rm(i)]
\item $\sin(t) \geq t/2$ if $0 \leq t \leq \pi/2$.
\item $\cos(t) \leq \exp(-t^2/4) \leq \exp(-2t^2/\pi^2)$ if $0 \leq t \leq \pi$.
\end{enumerate}
\end{lemma}

\begin{proof}
(i) Consider $f(t) := \sin(t)-t/2$. As $f'(t) = \cos(t)-1/2$, for $t \in [0,\pi/2]$, we have that
$$f(t) \geq \min \{f(0),f(\pi/2)\} = 0.$$

(ii) Now, again on $[0,\pi/2]$ for $g(t):= \exp(-t^2/4)-\cos(t)$ we have that $g(0) = 0$ and
$$g'(t) = \sin(t) - \frac{t}{2}\exp(-t^2/4) \geq f(t) \geq 0,$$
and so $g(t) \geq 0$. The statement is obvious if $t \in [\pi/2,\pi]$.
\end{proof}

\section{Proof of  main Theorems } \label{section:proof} 

We first prove the following auxiliary statement.

\begin{theorem} \label{theorem:main3} There is an integer $p_0 > 1$ such that the following statement holds whenever 
$p\ge p_0$. 
Let  $m \geq 43$ be an integer, $p \geq m$ a prime number, and let $n, s \ge 2$ be integers. 
Let  $V$ be a sequence  $A_1, \dots, A_n$ of elements of $S = PSL_m (p)$ of order at least $s$. 
Then 
$$\rho_V  \le  \frac{2} {p} +  \frac{120} {s} + \frac{19 }{n^{1/2} } .$$  
\end{theorem} 

All theorems from Section \ref{section:introduction} will follow from this theorem via some short arguments. 

\subsection{Proof of Theorem \ref{theorem:main3} }

We start with the identity \eqref{key}

\begin{equation} \label{key1}  \P (\prod_{i=1}^n \hat {A_i } = B) =  \frac{1}{|S|}\sum_{\Phi \in \Irr(S)}(\dim \Phi) 
  \trace\left( \prod_{i=1}^n \frac{ \Phi (A_i) + \Phi (A_i^{-1})  }{2} \Phi (B^{-1}) \right).
\end{equation} 

In what follows, we  may (and will) assume that $\Phi (X) $ is unitary for any element $X \in  PSL_m(p)$. As the absolute 
value of the trace of any unitary matrix does not exceed its size, we have 
$$|\trace \prod^n_{i=1}\Phi(A_i^{\pm 1})\Phi(B^{-1})| = |\trace \Phi(\prod^n_{i=1}A_i^{\pm 1}B^{-1})| \leq \dim \Phi.$$
Using Proposition \ref{pro:rep2}(i), we can easily bound  the contribution in the RHS of \eqref{key1} coming from representations of dimension less than 
$$d_0:=p^{(m^2-m-1)/2},$$ as the absolute value of this contribution is at most
\begin{equation}\label{small}
  \frac{1}{|S|}\sum_{\chi \in \Irr(G),~\chi(1) < d_0}\chi(1)^2  <  \frac{5}{3p}.
\end{equation} 

\vskip2mm

By Proposition \ref{pro:rep2}(iii), $k_i := \ord(A_i) < p^{m+1} < d_0$. 
Recall that $s \leq \min\{k_i \mid 1 \le i \le n\}$, the minimum order of the elements $A_i$. 


\vskip2mm 
Now we focus on representations $\Phi$ of dimension $d \geq d_0$. Choosing $p_0$ large enough, we may apply 
Proposition \ref{pro:rep2}(iv) to conclude that   the eigenvalues of 
$\Phi( A_i)$ are among the $k_i$-th roots of unity, with multiplicity at most 
\begin{equation}\label{for-m3}
  m_i := \lfloor \frac{3d}{k_i} \rfloor \le \frac{3d}{k_i}. 
\end{equation}  

As $\Phi (A_i)$ is unitary, by Lemma \ref{lemma:unitary}, any singular value of 
$$B_i = \frac{1}{2} \left( \Phi (A_i) + \Phi (A_i^{-1} )\right)$$ 
is either $1$ or $|\cos(2\pi j/k_i)|$ for some $1 \leq j \leq k_i-1$. Replacing $j$ by $k_i-j$ 
if necessary we may assume
that $1 \leq j \leq k_i/2$. Furthermore, if $k_i/4 < k_i \leq k_i/2$, we can replace 
$|\cos(2\pi j/k_i)|$ by $|\cos(\pi(k_i-2j)/k_i)|$ to get an angle in the range $[0,\pi/2]$. This procedure ensures 
by Lemma \ref{lemma:trig} 
that any singular value of $B_i$ is either $1$ or 
\begin{equation}\label{for-s5}
  \cos(\pi j/k_i) \leq \exp(-2j^2/k_i^2)
\end{equation}  
with $1 \leq j \leq k_i/2$, and each such occurs as a singular value of $B_i$ at most $4m_i$ times. 

Next we set 
\begin{equation}\label{for-l}
  l_0:= \left\lceil \frac{120d}{s} \right\rceil \ge  \frac{120d}{k_i} \geq 40m_i, 
\end{equation}  
where the inequalities follow from the condition on $s$ and \eqref{for-m3}.

Consider the matrix 
$$ M =  M(\Phi) = \prod_{i=1} ^n   \frac{\Phi (A_i)  + \Phi (A_i^{-1} ) }{2}\Phi(B^{-1})  
     = (\prod^n_{i=1}B_i)\Phi(B^{-1}). $$
Applying Lemma \ref{lemma:singular} repeatedly   and using the fact that $s_1(\Phi(B^{-1})) = 1$,  we have, for any $l$, that 
\begin{equation}\label{for-s6}
  \prod_ {j=1}^{l}  s_j  (M) \le  \prod^n_{i=1} \prod^{l}_{j=1}s_j(B_i).
\end{equation}  

Assume that $l > l_0$.  
As $l_0 \ge 40m_i$, we can write  $l = 4m_ia_i+b_i$ with integers $4m_i \leq b_i < 8m_i$ and $4 \le a_i$.  In fact,
$$a_i = \left\lfloor \frac{l-4m_i}{4m_i} \right\rfloor > \frac{l-8m_i}{4m_i} > \frac{l}{5m_i}.$$

\noindent
We are going  to show that  $\sigma_i := \prod^{l}_{j=1}s_j(B_i)$ is small, for all $1 \le i \le n$.
As shown above, the multiplicity of any singular value of $B_i$ is 
at most $4m_i <  l/5$. Furthermore, the different singular values of $B_i$  (in decreasing order) are bounded by 
$\exp(-2j^2/k_i^2)$ with $0 \leq j \leq k_i/2$ by \eqref{for-s5}. 
We can now bound $\sigma_i$ from the above by bounding the first $b_i$ singular values of $B_i$ by $1$
(note that we chose $b_i \geq 4m_i$), the next
$4m_i$ by $\exp(-2/k_i^2)$, the next $4m_i$ by $\exp(-2 \cdot 2^2/k_i^2)$, and so on. It follows that 
$\sigma_i \leq \exp(-E_i)$, where 
$$E_i = 4m_i \cdot \left(\frac{2}{k_i^2} +  \frac{2 \cdot 2^2}{k_i^2} + \ldots + \frac{2 \cdot a_i^2}{k_i^2}\right). $$

Since $\sum_{k=1}^{a_i} k^2 > a_i^3 /3 $ and $a_i > l/5m_i$, we obtain 

 $$E_i   > \frac{8m_ia_i^3}{3k_i^2} > \frac{8l^3}{375k_i^2m_i^2} \geq \frac{8l^3}{3375d^2} >  \frac{l^3}{422d^2},$$  
as $k_im_i  \leq 3d$ by \eqref{for-m3}. It follows  that

$$ \sigma_i \leq \exp(-l^3/422d^2).$$

Combining this estimate with \eqref{for-s6}, we obtain 
$$s_{l}(M)^{l} \leq \prod_ {j=1}^{l}  s_j  (M) \leq \prod^n_{i=1}\sigma_i \leq (\exp(-l^3/422d^2))^{n} = \exp(-nl^3/422d^2),$$
whence
$$s_{l} (M) \le \exp(-nl^2/422d^2)$$
for any $l > l_0$. 

Now we have 
$$\sum^d_{l = l_0+1}s_l(M) \leq \sum^d_{l=l_0+1}\exp\left(-\frac{nl^2}{422d^2}\right), $$ where the RHS is bounded from above by

$$     \int^\infty_{0}e^{-\frac{nt^2}{422d^2}}dt
    < \sqrt { \frac{422}{n}} d \int^\infty_0 e^{-t^2}dt < \frac{18.3d}{n^{1/2}}.$$

Using the trivial estimate $s_j(M) \leq 1$ for $1 \leq j \leq l_0$ and  Lemma \ref{lemma:singular}, we conclude that

$$|\Trace M | \le l_0 + \sum^d_{l = l_0+1}s_l(M)  < \lceil \frac{120d}{s} \rceil  +  \frac{18.3d}{n^{1/2}} \leq  \frac{120d}{s} +1 + \frac{18.3 d}{n^{1/2}} .$$

Thus, the contribution 
$$w (\Phi) =  (\dim\Phi)  \Trace \left(\prod_{i=1} ^n   \frac{\Phi (A_i)  + \Phi (A_i^{-1})  }{2}  \Phi (B^{-1} )\right) $$ 
is at most 
$$d^2 \cdot  \left( \frac{120}{s} + \frac{1}{d_0} + \frac{18.3 }{n^{1/2}}\right) $$
in absolute value, when $d =\dim \Phi$ is 
at least $d_0 = p^{(m^2-m-1)/2} > 2p$.

Using the identity $\sum_{\Phi \in \Irr(S)} (\dim \Phi )^2  = |S|$ and \eqref{small}, we conclude that  
the RHS in (\ref{key1}) is at most 
$$\frac{5}{3p}+    \frac{120}{s} + \frac{18.3}{n^{1/2}}  + \frac{1}{d_0} < \frac{2}{p} + \frac{120}{s} + \frac{19}{n^{1/2} }$$  
proving the theorem.

\subsection{Proof of Theorem \ref{theorem:main0}}

Replacing $A_i$ by $A^*_i := \diag(A_i,\det(A_i)^{-1},I_r)$ and 
$B^* := \diag(B,\det(B)^{-1},I_r)$ for a suitable $r \geq 1$, 
where $I_r$ denotes the identity $r \times r$-matrix, we can assume that $A^*_i, B^* \in SL_{m+r+1}(p)$ 
with $m+r+1 \geq 43$. Let
$\bar{A}_i$ and $\bar{B}$ denote the image of $A^*_i$, respectively of $B^*$, in $PSL_{m+r+1}(p)$.   
Note that the condition $r \geq 1$ implies that $\ord(A_i) = \ord(\bar{A}_i)$ for all $i$. Also, if 
$$A_1^{i_1}A_2^{i_2} \ldots A_n^{i_n} = B$$ 
for some $i_j = \pm 1$, $1 \leq j \leq n$, then
$$\bar{A}_1^{i_1}\bar{A}_2^{i_2} \ldots \bar{A}_n^{i_n} = \bar{B}.$$ 
Hence, when $p$ is sufficiently large,   
Theorem \ref{theorem:main3} applied to $PSL_{m+r+1}(p)$ implies the bound 
$$ \frac{2}{p} + \frac{120}{s} + \frac{19}{n^{1/2} }  \le 141 \cdot \max \left\{\frac{1}{s} , \frac{1}{n^{1/2} } \right \},$$  
given that $p \ge \min \{s, \sqrt{n} \}$. 

\subsection {Proof of Theorem \ref{theorem:main}} 

For any $B \in GL_m (\C)$ and the sequence $A_1, \dots, A_n$ of elements in $GL_m(\C)$, 
we use Corollary \ref{lemma:embed} to embed $A_1, \dots, A_n, B$ into $GL_m (p)$, for a large prime $p \ge m$. Theorem \ref{theorem:main} now follows from Theorem \ref{theorem:main0}. 

\subsection{Proof of Theorem \ref{theorem:main2}} 
If all $A_1, \dots, A_n$ have order at least $s$, then we can apply Theorem \ref{theorem:main}.
To handle the case when not every $A_i$  has order at least $s$, let us  revisit the setting of Theorem \ref{theorem:main3}.  Consider

$$ \P (\prod_{i=1}^n \hat {A_i } = B) = \sum_{\Phi \in \Irr(G)}  (\dim \Phi)  
     \trace \left( \prod_{i=1}^n \frac{ \Phi (A_i) + \Phi (A_i^{-1})  }{2} \Phi (B^{-1})\right).$$

We have

$$|\trace \prod_{i=1}^n \frac{ \Phi (A_i) + \Phi (A_i^{-1})  }{2} \Phi (B^{-1})| \le \sum_{i=1}^{\dim \Phi } s_i (M) $$ where $M$ is the product of $B_i : =(1/2)(\Phi (A_i) + \Phi (A_i^{-1}))$, $1 \leq i \leq n$.
Since $s_1 (B_i ) \le 1 $, by Lemma \ref{lemma:singular} we have

$$\sum_{i=1}^{\dim \Phi} s_i (M) \le \sum_{i=1}^{\dim \Phi}  s_ i (M')$$ where $M'$ is the product of all $B_i$ where $A_i$ has order at least $s$.  Now we can repeat the rest of the proof. 


\end{document}